\documentclass[12pt]{article}

\usepackage{url}
\usepackage[normalem]{ulem}
\usepackage{amsmath}
\newcommand{\stkout}[1]{\ifmmode\text{\sout{\ensuremath{#1}}}\else\sout{#1}\fi}\usepackage{amssymb}
\usepackage{amscd}
\usepackage{tikz}
\usetikzlibrary{arrows,automata,positioning}

\usepackage{amsmath}
\usepackage{amsthm}
\usepackage{amsfonts}
\usepackage{bm}
\usepackage[normalem]{ulem}

\usepackage{tikz}
\usetikzlibrary{arrows,automata,positioning}
\usetikzlibrary {arrows.meta}

\newtheorem{Theorem}{Theorem}
\newtheorem{lemma}[Theorem]{Lemma}
\newtheorem{corollary}[Theorem]{Corollary}

\theoremstyle{remark}
\newtheorem{remark}[Theorem]{Remark}

\numberwithin{Theorem}{section}

\begin{document}
\title{The analogue of overlap-freeness for the Fibonacci morphism}
\author{James D.\ Currie \& Narad Rampersad\\
Department of Mathematics \&
Statistics\\
The University of Winnipeg\\
{\tt j.currie@uwinnipeg.ca, n.rampersad@uwinnipeg.ca}}
\maketitle
\begin{abstract}
A $4^-$-power is a non-empty word of the form $XXXX^-$, where $X^-$ is obtained from $X$ by erasing the last letter. A binary word is called {\em faux-bonacci} if it contains no $4^-$-powers, and no factor 11. We show that faux-bonacci words bear the same relationship to the Fibonacci morphism that overlap-free words bear to the Thue-Morse morphism. We prove the analogue of Fife's Theorem for faux-bonacci words, and characterize the lexicographically least and greatest infinite faux-bonacci words.
\vspace{.1in}

\end{abstract}
\section{Introduction}
We study binary words, that is words over the alphabet ${\mathcal B}=\{0,1\}$. We use lower case letters (e.g., $w$) to denote finite words, and we  use 
bold-face letters to denote words with letters indexed by ${\mathbb N}$; e.g., $${\boldsymbol w}=w_1w_2w_3\cdots.$$ 
In the literature, words with letters indexed by ${\mathbb N}$ are variously referred to as $\omega$-words, infinite words, one-sided infinite words, etc. In this article we refer to them as {\em $\omega$-words}.
We freely use notions from combinatorics on words and from automata theory. Thus, for example, the set of finite words over ${\mathcal B}$ is denoted by ${\mathcal B}^*$, and the set of $\omega$-words  is denoted by ${\mathcal B}^\omega$.
We record morphisms inline, i.e., $g=[g(0),g(1)]$.

The binary overlap-free words constitute a classical object of study in combinatorics on words. They are particularly well understood because of their intimate connection to the Thue-Morse morphism $\mu=[01,10]$.
\begin{Theorem}\label{thue freeness} Let $w$ be binary word. Then $w$ is overlap-free if and only if $\mu(w)$ is overlap-free.
\end{Theorem}
Thue  \cite{thue12} proved that, for two-sided infinite words and for circular words, every overlap-free binary word arises as  the $\mu$ image of an overlap-free word. The analysis of finite words is more complicated, but these also arise via iterating $\mu$. (See, e.g.,  Restivo and Salemi \cite{restivo84}.)
\begin{Theorem}\label{thue factorization} Let $w\in{\mathcal B}^*$ be overlap-free. Then we can write $w=a\mu(u)b$, where $a, b\in\{\epsilon,0,00,1,11\}$, and $u$ is overlap-free. If $|w|\ge 7$ this factorization is unique. If ${\boldsymbol w}\in{\mathcal B}^\omega$ is overlap-free, then we can write ${\boldsymbol w}=a\mu({\boldsymbol u})$, for some overlap-free word ${\boldsymbol u}\in{\mathcal B}^\omega$ where $a\in\{\epsilon,0,00,1,11\}$.
\end{Theorem}

Characterizations of binary overlap-free words in terms of $\mu$ have allowed sharp enumerations of these words \cite{restivo84,kobayashi88,carpi93,cassaigne93,jungers09,guglielmi13}. These enumerations are closely connected to  a classical result known as Fife's Theorem  \cite{fife80}.
\begin{Theorem}[Fife's Theorem]
Let ${\boldsymbol u}\in{\mathcal B}^\omega$. Then ${\boldsymbol u}$ is overlap-free if and only if $${\boldsymbol u}=w\bullet{\boldsymbol f}$$
for some $w\in \{01,001\}$, and some ${\boldsymbol f}\in \{\alpha,\beta,\gamma\}^\omega$ containing no factor in $$I=(\alpha+ \beta)(\gamma\gamma)^*(\beta\alpha+ \gamma\beta+\alpha\gamma).$$
\end{Theorem}
In this theorem, each of $\alpha$, $\beta$, and $\gamma$ is an operator that maps a finite word ending in $\mu^n(0)$ or $\mu^n(1)$ to a finite word ending in $\mu^{n+1}(0)$ or $\mu^{n+1}(1).$ The second author's thesis \cite{rampersad07} contains a modern exposition of Fife's Theorem.

The Thue-Morse sequence ${\bf t}$ is a fixed point of $\mu$, namely, 
\[{\bf t}=\lim_{n\rightarrow\infty}\mu^n(0)\]
Due to Theorem~\ref{thue factorization}, ${\bf t}$ arises naturally in any study of overlap-free binary words. For example, Berstel \cite{berstel94} proved (See also Allouche {\em et al.} \cite{allouche98}):

\begin{Theorem}\label{lex}
The lexicographically greatest overlap-free binary $\omega$-word starting with 0 is 
${\bf t}$.
\end{Theorem}

Our rich understanding of binary overlap-free words comes from the strong connection between these words and the Thue-Morse morphism.  The thesis of the present article is that there exist analogous connections between other pairs of languages and morphisms. In a recent paper, the first author \cite{currie23} showed such a connection between the period-doubling morphism $\delta=[01,00]$ and {\em good words}. A binary word is {\em good} if it doesn't contain factors 11 or 1001, and doesn't encounter pattern 0000 or 00010100. He showed that:
\begin{itemize}
\item Good words factorize under $\delta$;
\item Word $\delta(w)$ is good if and only if $w$ is good;
\item An analog of Fife's Theorem holds for good $\omega$-words;
\item One can exhibit the lexicographically least and greatest good $\omega$-words.
\end{itemize}

Unfortunately, one may object that the period-doubling morphism does not give a proper `new' example of a language/morphism connection because of the 
close relationship of the period-doubling sequence ${\boldsymbol d}$ to the Thue-Morse sequence ${\boldsymbol t}$; it is well-known \cite{damanik00} that the period-doubling sequence can be obtained from the Thue-Morse sequence as follows: Let ${\boldsymbol {vtm}}$ be given by  $${\boldsymbol {vtm}}=g^{-1}({\boldsymbol t}),$$ where $g$ is the morphism on $\{0,1,2\}$ given by $g=[011,01,0]$. Then $${\boldsymbol d}=h({\boldsymbol {vtm}}),$$ where $h$ is the morphism on $\{0,1,2\}$ given by $h=[0,1,0]$. For this reason, in this article we consider another morphism, not connected to $\mu$ in the same way.

Another famous binary sequence is the Fibonacci word, which is the fixed point
\[{\boldsymbol \phi}=0100101001001010010100100101001001\cdots\]
of the binary morphism $\varphi$, where $\varphi=[01,0]$.  We call $\varphi$ the {\bf Fibonacci morphism}. The word ${\boldsymbol \phi}$ is central to the study of Sturmian words, and has a large literature. (See \cite{berstel86}, for example.)

For a non-empty word $X$, the word $X^-$ is obtained from $X$ by erasing its last letter. The word $^-X$ is obtained by erasing its first letter. We define a {\bf $4^-$-power} (said ``four minus''-power) to be a word of the form $XXXX^-$, some non-empty word $X$. Equivalently, a $4^-$-power is a word of period $p$, length $4p-1$ for some positive $p$. Extending periodically on the right or left with period $p$ by a single letter gives a fourth power. Thus a $4^-$-power $XXXX^-$ can also be written as $^-YYYY$, where $Y=aXa^{-1}$ and $a$ is the last letter of $X$.

A binary word is called {\em faux-bonacci} if it contains no factor 11 and no $4^-$-power. For the remainder of this paper we abbreviate `faux-bonacci' as `fb'. We will show that fb words bear the same relationship to the Fibonacci morphism that overlap-free words bear to the Thue-Morse morphism. We show that:
\begin{itemize} 
\item The fb words factorize under $\varphi$ (Theorem~\ref{main theorem}); \item Word $\varphi(w)$ is fb if and only if $w$ is fb (Theorem~\ref{phi fb}); 
\item An analog of Fife's Theorem holds for fb $\omega$-words (Theorem~\ref{fife010}); 
\item One can  exhibit the lexicographically least and greatest fb $\omega$-words (Theorem~\ref{lex least}).
\end{itemize}

These results raise the question of which well-studied properties of the Thue-Morse word ${\boldsymbol t}$  may generalize to the class of all morphic fixed points.
\section{Faux-bonacci words}

Unless otherwise specified, our words and morphisms are over the binary alphabet ${\mathcal B}=\{0,1\}$. 
\begin{lemma}\label{phi^-1 fb}
Let ${\boldsymbol u}\in{\mathcal B}^\omega$. Suppose $\varphi({\boldsymbol u})$ is fb. Then ${\boldsymbol u}$ is fb. \end{lemma}
\begin{proof} We prove the contrapositive: Suppose ${\boldsymbol u}$ is not fb; we prove that $\varphi({\boldsymbol u})$ is not fb. 

If 11 is a factor of ${\boldsymbol u}$, then one of $\varphi(110)=0001$ and $\varphi(111)=000$ is a factor of $\varphi({\boldsymbol u})$. Each of these contains the $4^-$-power 000. 

Suppose $u$ contains an $4^-$-power $XXXX^-$ where $X=xa$ for some $a\in{\mathcal B}$. If $a=1$, then $\varphi(u)$ contains $\varphi(x1x1x1x)=
\varphi(x)0\varphi(x)0\varphi(x)0\varphi(x)$, which is an $4^-$-power. If $a=0$, then $u$ contains one of $x0x0x0x0$ and $x0x0x0x1$. In either case, $\varphi(u)$ contains the $4^-$-power $\varphi(x)01\varphi(x)01\varphi(x)01\varphi(x)0$.
\end{proof}

\begin{remark} \label{no 11} Let $w=w_1w_2w_3\cdots w_n$ with $w_i\in{\mathcal B}$ and suppose $|w|_{11}=0$. There is a unique word $u$ such that $0w=\varphi(u)$ for some $u$. 
Let ${\boldsymbol w}\in{\mathcal B}^\omega$ and suppose $|{\boldsymbol w}|_{11}=0$. Then we can write ${\boldsymbol w}=a\varphi({\boldsymbol u})$, some ${\boldsymbol u}\in{\mathcal B}^\omega$ where $a\in\{\epsilon,1\}$.
\end{remark}

\begin{Theorem}\label{main theorem}
Let ${\boldsymbol w}\in{\mathcal B}^\omega$ be fb. Then we can write ${\boldsymbol w}=a\varphi({\boldsymbol u})$, where ${\boldsymbol u}$ is fb and where $a\in\{\epsilon,1\}$.
\end{Theorem}
\begin{proof}This is immediate from Lemma~\ref{phi^-1 fb} and Remark~\ref{no 11}.\end{proof}

\begin{Theorem}\label{phi fb}
Let ${\boldsymbol w}\in{\mathcal B}^\omega$. Then $\varphi({\boldsymbol w})$ is fb if and only if ${\boldsymbol w}$ is fb.
\end{Theorem}
\begin{proof} The only if direction is Lemma~\ref{phi^-1 fb}. Suppose then that ${\boldsymbol w}$ is fb. Certainly $\varphi({\boldsymbol w})$ cannot contain 11 as a factor.
Further, if 000 is a factor of $\varphi({\boldsymbol w})$, then one of 110 and 111 is a factor of ${\boldsymbol w}$; however, ${\boldsymbol w}$ is fb, so this is impossible.
 
Suppose that $\varphi({\boldsymbol w})$ has a factor $XXXx$ where $X=xa$, some $a\in{\mathcal B}$. 

If $a=1$ then $XX=x1x1$. Since 11 is not a factor of $\varphi({\boldsymbol w})$, $x$ must be non-empty and have first letter 0. Write $X=\varphi(Y0)$. Then $XXXx=\varphi(Y0Y0Y0Y)0$, and ${\boldsymbol w}$ contains the overlap  $Y0Y0Y0Y$. This is a contradiction, since ${\boldsymbol w}$ is fb.

Assume then that $a=0$. If the first letter of $X$ is 1, then the factor 
$XXXx=x0x0x0x$ of $\varphi(w)$ must appear in the context $0x0x0x0x$. Word $x$ cannot be empty, since 000 is not a factor of $\varphi({\boldsymbol w})$.
If the last letter of $x$ is 1, then replacing $X$ by $0x$ reduces to the previous case. Suppose then that the last letter of $x$ is 0. Write $x=x'0$. Then $XXXx=x'00x'00x'00x'0$. Since 000 is not a factor of $\varphi({\boldsymbol w})$, $x'$ is non-empty and starts and ends with 1. This implies that $\varphi({\boldsymbol w})$ contains the factor
$0XXXx'=0x'00x'00x'00x'$. Write $0x'0=\varphi(Y1)$. Then
$0x'00x'00x'00x'=\varphi(Y1Y1Y1Y)$, and the $4^-$-power $Y1Y1Y1Y$ is a factor of ${\boldsymbol w}$. This is impossible.
\end{proof}

\begin{corollary} The Fibonacci word ${\boldsymbol \phi}$ is fb.
\end{corollary}

\begin{lemma}\label{internal 10101} Suppose that ${\boldsymbol w}\in{\mathcal B}^\omega$ is fb. Then 10101 is not a factor of $^-{\boldsymbol w}$.
\end{lemma}
\begin{proof} Suppose 10101 is  a factor of $^-{\boldsymbol w}$. Since 11 is not a factor of ${\boldsymbol w}$, extending 10101 to the left and right we find that 0101010 is a factor of ${\boldsymbol w}$. But 0101010 is a $4^-$-power.
\end{proof}

\begin{lemma}\label{10101} Suppose $0101{\boldsymbol w}\in{\mathcal B}^\omega$ is fb. Then $10101{\boldsymbol w}$ is fb.
\end{lemma}
\begin{proof}
If $10101{\boldsymbol w}$ is not fb it must begin with a $4^-$-power with some positive period $p$ and length $4p-1$. If $p\ge 5$, then 10101 is a factor of the fb word $^-0101{\boldsymbol w}$. This is impossible by Lemma~\ref{internal 10101}, so that $p\le 4$. Thus $p$ is a period of 10101, so that $p$ is 2 or 4. This forces 1010101 to be a prefix of $10101{\boldsymbol w}$, and again 10101 is a factor of the fb word $^-0101{\boldsymbol w}$.
\end{proof}
\section{An analogue of Fife's Theorem}
Let $U$ be the set of infinite fb words. For $u\in {\mathcal B}^*$, let $U_u=U\cap u{\mathcal B}^\omega$.
\begin{lemma}\label{allouche} Let ${\boldsymbol v}\in{\mathcal B}^\omega$.
\begin{enumerate}
\item[(i)] $\text{Word }\varphi({\boldsymbol v})\in U\iff {\boldsymbol v}\in U;$
\item[(ii)] $\text{Word }1\varphi({\boldsymbol v})\in U\iff 0{\boldsymbol v}\in U\text{ or } {\boldsymbol v}\in U_{00}.$
\end{enumerate}
\end{lemma}
\begin{proof} Part (i) is Theorem~\ref{phi fb}. For part (ii), first suppose that 
$0{\boldsymbol v}\in U$ or ${\boldsymbol v}\in U_{00}.$
If $0{\boldsymbol v}\in U$, then by Theorem~\ref{phi fb} it follows that $01\varphi({\boldsymbol v})=\varphi(0{\boldsymbol v})$ is fb, so in particular $1\varphi({\boldsymbol v})$ is fb; if ${\boldsymbol v}\in U_{00}$, then $\varphi({\boldsymbol v})$ is fb by Theorem~\ref{phi fb} and has prefix $0101$, so that $1\varphi({\boldsymbol v})$ is fb by Lemma~\ref{10101}.

In the other direction, suppose that $1\varphi({\boldsymbol v})\in U$. To get a contradiction, suppose that $0{\boldsymbol v}\not\in U$ and ${\boldsymbol v}\not\in U_{00}.$ Since $1\varphi({\boldsymbol v})\in U$, we must have $\varphi({\boldsymbol v})\in U$, so that ${\boldsymbol v}\in U$ by Theorem~\ref{phi fb}. From $0{\boldsymbol v}\not\in U$ we deduce that a prefix of $0{\boldsymbol v}$ is not fb. Since $11$ cannot be a prefix of $0{\boldsymbol v}$, we deduce that $0{\boldsymbol v}$ has a prefix of the form $XXXX^-$ for some non-empty word $X$. It follows that $0$ is a prefix of $X$. Since ${\boldsymbol v}\not\in U_{00}$, we conclude that $|X|\ge 2$. Since $XX$ is a factor of the fb word ${\boldsymbol v}$, and $XX$ has prefix $X0$, we conclude that $00$ is not a suffix of $X$; otherwise ${\boldsymbol v}$  would have factor $000=0000^-$, which is impossible.

Let ${\boldsymbol v}$ have prefix $^-XXXX^-a$, where $a\in{\mathcal B}$. Letter $a$ cannot be the last letter of $X$, or ${\boldsymbol v}$ would start with the $4^-$-power $^-XXXX$. It follows that $X^-$ is followed in $^-XXXX^-a$ variously by $a$ and by the other letter of ${\mathcal B}$. This implies that the last letter of $X^-$ is $0$, since $1$ cannot be followed by $1$ in ${\boldsymbol v}$. Since $00$ is not a suffix of $X$, word $X$ must end in $01$.

We cannot, however, have $X=01$. In this case, word ${\boldsymbol v}$ would have prefix $^-XXXX^-0=1 01 01 0 0$, causing $1\varphi({\boldsymbol v})$ to have prefix $100100100101$, which begins with the $4^-$-power  $10010010010$. It follows that $|X|\ge 3$. Since $X^-$ is sometimes followed by $0$ in ${\boldsymbol v}$, word $X^-$ cannot have suffix $00$, since $000$ is not a factor of ${\boldsymbol v}$. As the last letter of $X^-$ is $0$, this implies that $X^-$ ends in $10$. Now $X$ ends in $101$, and $XX$ is a factor of ${\boldsymbol v}$. Thus Lemma~\ref{internal 10101} implies that $X$ does not start $01$. It follows that $X$ starts $00$. Write $X=00Y101$ and ${\boldsymbol v}=^-XXXX^-0{\boldsymbol u}$. Then
$$1\varphi({\boldsymbol v})=101\varphi(Y)0010 0101\varphi(Y)0010 0101\varphi(Y)0010 0101\varphi(Y)001 01 \varphi({\boldsymbol u})
$$
which begins with the $4^-$-power $ZZZZ^-$ where $Z=101\varphi(Y)0010 0$. This is a contradiction.
\end{proof}

Consider the {\em finite Fibonacci words} $F_n$ defined for non-negative integers $n$ by 
\begin{align*}
F_0&=0,\\
F_1&=01, \text{ and}\\
F_{n+2}&=F_{n+1}F_n\text{ for }n\ge 0.
\end{align*} 

The following lemma is proved by induction.

\begin{lemma}\label{n at least 2} Suppose that $u$ is a binary word and $p$ is a prefix of $\varphi(u)$. If $p$ ends in $F_{n+1}$, some $n\ge 2$, then $p=\varphi(q)$, where $q$ is a prefix of $u$ ending in $F_n$.
\end{lemma}
\begin{remark} The condition $n\ge 2$ is necessary. If $u=00$, then $F_2$ is a factor of $\varphi(u)$, but $F_1$ is not a factor of $u$.
\end{remark}

We use the notation $\pi(w)$ for the {\em Parikh vector} of a binary word. Thus
$$\pi(w)=[|w|_0,|w|_1].$$

 Suppose that $w$ is a word of the form $y_nF_n$ where $\pi(y_n)\le \pi(F_n)-\pi(0)$. We define operations on $w$ by
\begin{align*}
\alpha(w)&=y_nF_{n+1}\\
\beta(w)&=y_nF_{n-1}F_{n+1}.
\end{align*}

One checks that $F_n$ is a prefix of $F_{n-1}F_{n+1}$, so that $w$ is always a prefix of $\alpha(w)$ and $\beta(w)$.
Because $\pi(y_n)\le \pi(F_n)-\pi(0)$, we have that $\pi(y_nF_{n-1})\le
\pi(F_nF_{n-1})-\pi(0)=\pi(F_{n+1})-\pi(0)$, so that we can iterate the maps $\alpha$ and $\beta$. Let ${\mathcal O}=\{\alpha,\beta\}$. We define operators $w\bullet f$ for $f\in{\mathcal O}^*$ by
\begin{align*}
w\bullet\epsilon&=w\\
w\bullet (f\gamma)&=(w\bullet f)\bullet\gamma \text{ for }\gamma\in{\mathcal O}.
\end{align*}
\begin{remark}
If $g$ is a prefix of $f$, then $w\bullet g$ will be a prefix of $w\bullet f$.
\end{remark}
\begin{lemma} Let ${\boldsymbol u}$ be a fb $\omega$-word, and for $n\ge 2$, let $w_n=y_nF_n$ be the shortest prefix of ${\boldsymbol u}$ ending in $F_n$. For $n\ge 2$, we have $\pi(y_n)\le \pi(F_n)-\pi(0)$.
\end{lemma}
\begin{proof}
Word $w_2$ is a fb word containing $F_2=010$ exactly once, as a suffix. The candidates  are $010$, $0010$, $1010$, and $10010$, which would yield $y_2=\epsilon$, $0$, $1$, and $10$, respectively. Thus the result is true for $n=2$. 

Suppose the result has been found to be true for $n=k$, some $k\ge 2$. Write ${\boldsymbol u}=a\varphi({\boldsymbol u'})$ where ${\boldsymbol u'}$ is fb, and $a\in\{\epsilon,1\}$. 
Since the first letter of $F_{k+1}$ is $0$ and $a\in\{\epsilon,1\}$, any occurrence of 
$F_{k+1}$ in ${\boldsymbol u}$ starts in $\varphi({\boldsymbol u'})$. By Lemma~\ref{n at least 2}, any prefix $qF_{k+1}$ of ${\boldsymbol u}$ has the form $a\varphi(q'F_k)$ where $q'F_k$ is a prefix of ${\boldsymbol u'}$. Thus 
$w_{k+1}=
a\varphi(w'_k)$ where $w'_k =y'_kF_k$ is the shortest prefix of ${\boldsymbol u'}$ ending in $F_k$. By the induction hypothesis, $\pi(y'_k)\le\pi(F_k)-\pi(0)$. It follows that
\begin{align*}
\pi(a\varphi(y'_k))&\le \pi(1)+\pi(\varphi(F_k))-\pi(\varphi(0))\\
&=\pi(1)+\pi(F_{k+1})-\pi(01)\\
&=\pi(F_{k+1})-\pi(0).
\end{align*}
\end{proof}
\begin{remark} Again, the condition $n\ge 2$ is necessary. If $w_2=1001$, then $y_2=10$, and $\pi(y_2)\not\le\pi(F_2)-\pi(0)$.
\end{remark}

\begin{lemma} Let ${\boldsymbol u}$ be a fb $\omega$-word, and for $n\ge 2$, let $w_n=y_nF_n$ be the shortest prefix of ${\boldsymbol u}$ ending in $F_n$. For $n\ge 2$, we have $w_{n+1}\in\{w_n\bullet\alpha,w_n\bullet\beta\}$.
\end{lemma}
\begin{proof}
The only fb words starting with $F_2$, and containing $F_3$ exactly once, as a suffix, are $01001=010\bullet\alpha$ and $0101001=010\bullet\beta.$ Thus the result is true for $n=2$.

Suppose the result has been found to be true for $n=k$, some $k\ge 2$. Write ${\boldsymbol u}=a\varphi({\boldsymbol u'})$ where ${\boldsymbol u'}$ is fb, and $a\in\{\epsilon,1\}$. For each $n$, let $w'_n=y'_nF_n$ be the shortest prefix of ${\boldsymbol u'}$ ending in $F_n$. By the induction hypothesis, $w'_{k+1}$ is either
$w'_k\bullet\alpha = y'_kF_{k+1}$ or
$w'_k\bullet\beta = y'_kF_{k-1}F_{k+1}.$ As in the previous proof, 
$w_{k+2}=a\varphi(w'_{k+1})$ and $y_{k+2}=a\varphi(y'_{k+1})$. Then
\begin{align*}
w_{k+2}&=a\varphi(w'_{k+1})\\
&\in\{a\varphi(y'_kF_{k+1}),a\varphi(y'_kF_{k-1}F_{k+1})\}\\
&=\{y_{k+1}F_{k+2},y_{k+1}F_{k}F_{k+2}\}\\
&=\{w_{k+1}\bullet\alpha,w_{k+1}\bullet\beta\}.
\end{align*}
\end{proof}
\begin{corollary}\label{growing u} Let ${\boldsymbol u}$ be a fb $\omega$-word. There is an $\omega$-word ${\boldsymbol f}=\Pi_{n=2}^\infty f_n$, $f_n\in{\mathcal O}$, and a `seed' word $w_2\in\{010,0010,1010,10010\}$, such that 
$${\boldsymbol u}=\lim_{n\rightarrow\infty}w_2\bullet(f_2f_3\cdots f_n).$$
\end{corollary}
\begin{proof} Letting the $w_n$ be as in the previous lemma, choose 
$f_n$ such that $w_{n+1}=w_n\bullet f_n$.
\end{proof}
Suppose that $w$ is a word of the form $y_nF_n$ where $\pi(y_n)\le \pi(F_n)-\pi(0)$. We note that 
\begin{align*}
&w\bullet\alpha=wF_n^{-1}F_{n+1}=wF_n^{-1}\varphi^{n-2}(F_2\bullet\alpha),\\
&w\bullet\beta=wF_n^{-1}F_{n-1}F_{n+1}=wF_n^{-1}\varphi^{n-2}(F_2\bullet\beta).
\end{align*}
If $f\in{\mathcal O}^k$, write $f=f_1f_2\cdots f_k$, where each $f_i\in{\mathcal O}$.  Suppose that $w$ is a word of the form $y_2F_2$ where $\pi(y_2)\le \pi(F_2)-\pi(0)$. By induction, 
$w\bullet f$ is a word of the form $y_{k+2}F_{k+2}$ where $\pi(y_{k+2})\le \pi(F_{k+2})-\pi(0)$. 
We find that
\begin{align}\label{parse f}
&~~~~~w\bullet f\nonumber \\
&=(w\bullet f_1f_2\cdots f_{k-1})\bullet f_k\nonumber \\
&=(w\bullet f_1f_2\cdots f_{k-1})F_{k+1}^{-1}\varphi^{k-1}(F_2\bullet f_k)\nonumber \\
&=(w\bullet f_1f_2\cdots f_{k-2})F_{k}^{-1}\varphi^{k-2}(F_2\bullet f_{k-1})F_{k+1}^{-1}\varphi^{k-1}(F_2\bullet f_k)\nonumber \\
&~~~~\vdots\nonumber \\
&=w\Pi_{j=1}^k F_{j+1}^{-1}\varphi^{j-1}(F_2\bullet f_j)
\end{align}

Let ${\boldsymbol f}\in{\mathcal O}^\omega$,
${\boldsymbol f}=f_1f_2f_3\cdots$, where each $f_i\in{\mathcal O}$. For $w\in\{010,0010,1010,10010\}$, define 
$$w\bullet{\boldsymbol f}=\lim_{n\rightarrow\infty}w\bullet(f_1f_2f_3\cdots f_n).$$

\begin{lemma}\label{parse lemma} Suppose $w\in\{010,0010,1010,10010\}$ and ${\boldsymbol f}\in{\mathcal O}^\omega$ and $g\in{\mathcal O}^k$. Suppose 
$F_2\bullet{\boldsymbol f}={\boldsymbol x}$. Then
\begin{equation}\label{parse}w\bullet(g{\boldsymbol f})=(w\bullet g) F_{k+2}^{-1}\varphi^k({\boldsymbol x}).\end{equation}
\end{lemma}
\begin{proof} Write ${\boldsymbol f}=f_1f_2f_3\cdots$ and $g=g_1g_2\cdots g_k$ where the $f_i$, $g_i\in{\mathcal O}$. Suppose that
$F_2\bullet f_1\cdots f_n = x_n$. We will show that
$w\bullet(g_1g_2\cdots g_kf_1f_2\cdots f_n)=(w\bullet g) F_{k+2}^{-1}\varphi^k(x_n)$, and the result follows by taking limits.
From (\ref{parse f}),  
$$x_n=F_2\bullet f_1\cdots f_n = F_2\Pi_{j=1}^n F_{j+1}^{-1}\varphi^{j-1}(F_2\bullet f_j)$$
and 
\begin{align*}
&~~~~~w\bullet(g_1g_2\cdots g_kf_1f_2\cdots f_n)\\
&=w\Pi_{j=1}^k F_{j+1}^{-1}\varphi^{j-1}(F_2\bullet g_j)
\Pi_{j=1}^n F_{k+j+1}^{-1}\varphi^{k+j-1}(F_2\bullet f_j)\\
&=(w\bullet g)
\Pi_{j=1}^n \varphi^k(F_{j+1})^{-1}\varphi^k(\varphi^{j-1}(F_2\bullet f_j))\\
&=(w\bullet g)
\varphi^k(F_2^{-1}F_2\Pi_{j=1}^n F_{j+1}^{-1}\varphi^{j-1}(F_2\bullet f_j))\\
&=(w\bullet g)
\varphi^k(F_2)^{-1}\varphi^k(F_2\Pi_{j=1}^n F_{j+1}^{-1}\varphi^{j-1}(F_2\bullet f_j)))\\
&=(w\bullet g)F_{k+2}^{-1}
\varphi^k(x_n),
\end{align*}
as desired.
\end{proof}

Define sets $F$ and $V$ by
\[F =\alpha^*\beta(\alpha\alpha+\alpha\beta+\beta\alpha\alpha+\beta\beta\beta^*\alpha\alpha)^*(\beta\alpha\beta+\beta\beta\beta^*\alpha\beta), \]
\[V = {\mathcal O}^\omega-{\mathcal O}^*F{\mathcal O}^\omega.\]
\begin{Theorem}\label{fife010}
Let ${\boldsymbol x}\in{\mathcal B}^\omega$.
If ${\boldsymbol x}$ begins with $010$, then ${\boldsymbol x}$ is fb if and only if ${\boldsymbol x}= 010\bullet {\boldsymbol f}$ for some
${\boldsymbol f}\in V$.
\end{Theorem}
The set $F$ consists of forbidden factors for words of $V$. 

Let
$W = \{{\boldsymbol f}\in {\mathcal O}^\omega
: 010\bullet{\boldsymbol f}\in U\}$. 
To prove Theorem~\ref{fife010} it is enough to prove that $W = V$.
Let $L\subseteq{\mathcal O}^\omega$
 and let $x\in{\mathcal O}^*$. We define the (left) quotient $x^{-1}L$ by
$x^{-1}L= \{{\boldsymbol y} \in  {\mathcal O}^\omega 
: x{\boldsymbol y} \in  L\}$.
The next lemma establishes several
identities concerning quotients of the set $W$. They are proved using (\ref{parse}) and Lemma~\ref{allouche}. The identities demonstrate that $W$ is precisely the set of infinite labeled paths through the automaton ${A_{010}}$ given in
Figure~\ref{automaton}. These are just the labeled paths omitting factors in $F$, so that  $W = V$. Thus, proving Lemma~\ref{identities} establishes Theorem~\ref{fife010}.

\begin{lemma}\label{identities}The following identities hold:
\begin{enumerate}\item[(a)]$W = \alpha^{-1}W$;
\item[(b)]$\beta^{-1}W=(\beta\alpha\alpha)^{-1}W=(\beta\alpha\beta)^{-1}W^{-1}=(\beta\beta\alpha\alpha)^{-1}W$;
\item[(c)]$(\beta\beta\alpha)^{-1}W=(\beta\beta\beta\alpha)^{-1}W$;
\item[(d)]$(\beta\beta\beta)^{-1}W=(\beta\beta\beta\beta)^{-1}W$;
\item[(e)]$(\beta\beta\alpha\beta)^{-1}W=\emptyset$.
\end{enumerate}
\end{lemma}
Each set of identities corresponds to the state of ${A_{010}}$ with the same label as the identities. There are two further states: $(f)$, corresponding to $(\beta\alpha)^{-1}W$, and $(g)$, corresponding to $(\beta\beta)^{-1}W$
The non-accepting sink $(e)$ is not shown in the figure. The automaton $A_{010}$ is not minimal; for example, $(d)$ can be identified with $(g)$. However, this form highlights parallels between it and the three automata we present later.

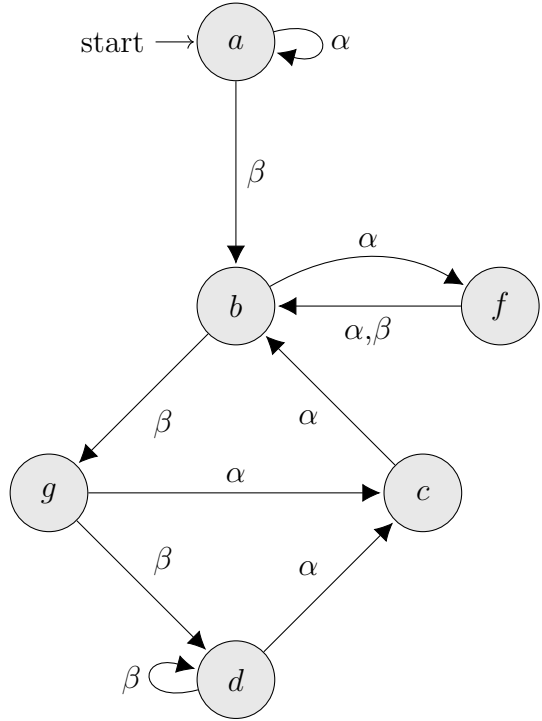
\begin{figure}
\begin{center}
\begin{tikzpicture}[shorten >=1pt,node distance=100pt,auto]

  \tikzstyle{every state}=[fill={rgb:black,1;white,10}]

  \node[state,initial]   (a)                      {$a$};
  \node[state] (b) [below of=a]  {$b$};
  \node[state] (f) [right of=b] {$f$};
  \node[state] (g) [below left of=b]     {$g$};
  \node[state] (c) [below right of=b]     {$c$};
  \node[state] (d) [below left of=c]     {$d$};
  \path[->,arrows = {-Latex[width=7pt, length=7pt]},every loop/.style={arrows = {-Latex[width=7pt, length=7pt]}}]

  (a)   edge [loop right]    node {$\alpha$} (a)
       edge  node {$\beta$} (b)
  (b) edge [bend left] node {$\alpha$} (f)
       edge  node {$\beta$} (g)
  (c) edge   node {$\alpha$} (b)
  (d) edge  node {$\alpha$} (c)
       edge [loop left]    node {$\beta$} (d)  
  (f) edge node {$\alpha$,$\beta$} (b)
  (g) edge  node {$\alpha$} (c)
       edge  node {$\beta$} (d)
 ;
\end{tikzpicture}
\caption{`Fife' automaton ${A_{010}}$ for $U$.}
\label{automaton}
\end{center}
\end{figure}
\begin{proof}\begin{enumerate}
\item[(a)] We have \begin{align*}
&\alpha{\boldsymbol f}\in W\\
\iff& 010\bullet\alpha{\boldsymbol f}\in U\\
\iff&(010\bullet\alpha)\varphi(F_2)^{-1}\varphi({\boldsymbol x})\in U\text{ by (\ref{parse})}\\
\iff& \stkout{01001(01001)}^{-1}\varphi( {\boldsymbol x})\in U\\
\iff& {\boldsymbol x}\in U\text{ by Lemma~\ref{allouche} (i)}\\
\iff& 010\bullet{\boldsymbol f}\in U\\
\iff& {\boldsymbol f}\in W,
\end{align*} 
so that $\alpha^{-1}W=W$.

\item[(b)] Here \begin{align*}
&\beta{\boldsymbol f}\in W\\
\iff& 010\bullet\beta{\boldsymbol f}\in U\\
\iff& (010\bullet\beta)\varphi(F_2)^{-1}\varphi({\boldsymbol x})\in U\\
\iff& 01~\stkout{01001(01001)}^{-1}\varphi({\boldsymbol x})\in U\\
\iff& \varphi(0{\boldsymbol x})\in U\\
\iff& 0{\boldsymbol x}\in U\text{ by Lemma~\ref{allouche} (i)}.
\end{align*}

Similarly we find that \begin{align*}
&\beta\alpha\alpha{\boldsymbol f}\in W\\
\iff& 010\bullet\beta\alpha\alpha{\boldsymbol f}\in U\\
\iff& (010\bullet\beta\alpha\alpha)\varphi^3(F_2)^{-1}\varphi^3({\boldsymbol x})\in U\\
\iff& 01~\stkout{0100101001001(0100101001001)^{-1}}\varphi^3({\boldsymbol x})\in U\\
\iff& \varphi(0\varphi^2({\boldsymbol x}))\in U\\
\iff& 0\varphi^2({\boldsymbol x})\in U\text{ by Lemma~\ref{allouche} (i)}\\
\iff& \varphi(1\varphi({\boldsymbol x}))\in U\\
\iff& 1\varphi({\boldsymbol x})\in U\text{ by Lemma~\ref{allouche} (i)}\\
\iff& 0{\boldsymbol x}\in U\text{ or }{\boldsymbol x}\in U_{00}\text{ by Lemma~\ref{allouche} (ii)}\\
\iff& 0{\boldsymbol x}\in U.
\end{align*}
Here we use the fact that $01$ is a prefix of ${\boldsymbol x}$, so that  ${\boldsymbol x}\not\in U_{00}$. 

Again,\begin{align*}
&\beta\alpha\beta{\boldsymbol f}\in W\\
\iff& 010\bullet\beta\alpha\beta{\boldsymbol f}\in U\\
\iff& (010\bullet\beta\alpha\beta)\varphi^3(F_2)^{-1}\varphi^3({\boldsymbol x})\in U\\
\iff& 01 01001 ~\stkout{0100101001001(0100101001001)^{-1}}\varphi^3({\boldsymbol x})\in U\\
\iff& \varphi(0010\varphi^2({\boldsymbol x}))\in U\\
\iff& \varphi(101\varphi({\boldsymbol x}))\in U\\
\iff& 101\varphi({\boldsymbol x})\in U\\
\iff& 1\varphi(0{\boldsymbol x})\in U\\
\iff& 00{\boldsymbol x}\in U\text{ or }0{\boldsymbol x}\in U_{00}\\
\iff& 0{\boldsymbol x}\in U.
\end{align*}
Here we use the fact that $0$ is a prefix of ${\boldsymbol x}$, so that $00$ is a prefix of  $0{\boldsymbol x}$. 

Finally we get 
\begin{align*}
&\beta\beta\alpha\alpha{\boldsymbol f}\in W\\
\iff& 010\bullet\beta\beta\alpha\alpha{\boldsymbol f}\in U\\
\iff& (010\bullet\beta\beta\alpha\alpha)\varphi^4(F_2)^{-1}\varphi^4({\boldsymbol x})\in U\\
\iff& 01 010 \stkout{0100101001001 01001010(0100101001001 01001010)^{-1}}\varphi^4({\boldsymbol x})\in U\\
\iff& \varphi(001\varphi^3({\boldsymbol x}))\in U\\
\iff& 001\varphi^3({\boldsymbol x})\in U\\
\iff& \varphi(10\varphi^2({\boldsymbol x}))\in U\\
\iff& 10\varphi^2({\boldsymbol x})\in U\\
\iff& 1\varphi(1\varphi({\boldsymbol x}))\in U\\
\iff& 01\varphi({\boldsymbol x})\in U\text{ or }1\varphi({\boldsymbol x})\in U_{001}\\
\iff& 01\varphi({\boldsymbol x})\in U\\
\iff& \varphi(0{\boldsymbol x})\in U\\
\iff& 0{\boldsymbol x}\in U.
\end{align*}
Thus $\beta^{-1}W=(\beta\alpha\alpha)^{-1}W=(\beta\alpha\beta)^{-1}W=(\beta\beta\alpha\alpha)^{-1}W$, as desired.

\item[(c)] Here \begin{align*}
&\beta\beta\alpha{\boldsymbol f}\in W\\
\iff& 010\bullet\beta\beta\alpha{\boldsymbol f}\in U\\
\iff& (010\bullet\beta\beta\alpha)\varphi^3(F_2)^{-1}\varphi^3({\boldsymbol x})\in U\\
\iff& 01010~\stkout{0100101001001(0100101001001)^{-1}}\varphi^3({\boldsymbol x})\in U\\
\iff& \varphi(001\varphi^2({\boldsymbol x}))\in U\\
\iff& 001\varphi^2({\boldsymbol x})\in U\\
\iff& 10\varphi({\boldsymbol x})\in U\\
\iff& 01{\boldsymbol x}\in U\text{ or }1{\boldsymbol x}\in U_{00}\\
\iff& 01{\boldsymbol x}\in U.
\end{align*}

Similarly,
\begin{align*}
&\beta\beta\beta\alpha{\boldsymbol f}\in W\\
\iff& 010\bullet\beta\beta\beta\alpha{\boldsymbol f}\in U\\
\iff& (010\bullet\beta\beta\beta\alpha)\varphi^4(F_2)^{-1}\varphi^4({\boldsymbol x})\in U\\
\iff& 01 010 01001\stkout{0100101001001 01001010(0100101001001 01001010)^{-1}}\varphi^4({\boldsymbol x})\in U\\
\iff& \varphi(001010\varphi^3({\boldsymbol x}))\in U\\
\iff& 001010\varphi^3({\boldsymbol x})\in U\\
\iff& \varphi(1001\varphi^2({\boldsymbol x}))\in U\\
\iff& 1001\varphi^2({\boldsymbol x})\in U\\
\iff& 1\varphi(10\varphi({\boldsymbol x}))\in U\\
\iff& 010\varphi({\boldsymbol x})\in U\text{ or }10\varphi({\boldsymbol x})\in U_{001}\\
\iff& 010\varphi({\boldsymbol x})\in U\\
\iff& \varphi(01{\boldsymbol x})\in U\\
\iff& 01{\boldsymbol x}\in U.
\end{align*}
Thus $(\beta\beta\alpha)^{-1}W=(\beta\beta\beta\alpha)^{-1}W$, as desired.
\item[(d)] We have \begin{align*}
&\beta\beta\beta{\boldsymbol f}\in W\\
\iff& 010\bullet\beta\beta\beta{\boldsymbol f}\in U\\
\iff& (010\bullet\beta\beta\beta)\varphi^3(F_2)^{-1}\varphi^3({\boldsymbol x})\in U\\
\iff& 01 010 01001~\stkout{0100101001001(0100101001001)^{-1}}\varphi^3({\boldsymbol x})\in U\\
\iff& \varphi(001010\varphi^2({\boldsymbol x}))\in U\\
\iff& 001010\varphi^2({\boldsymbol x})\in U\\
\iff& 1001\varphi({\boldsymbol x})\in U\\
\iff& 010{\boldsymbol x}\in U\text{ or }10{\boldsymbol x}\in U_{00}\\
\iff& 010{\boldsymbol x}\in U.
\end{align*}

Similarly,
\begin{align*}
&\beta\beta\beta\beta{\boldsymbol f}\in W\\
\iff& 010\bullet\beta\beta\beta\alpha{\boldsymbol f}\in U\\
\iff& (010\bullet\beta\beta\beta\alpha)\varphi^4(F_2)^{-1}\varphi^4({\boldsymbol x})\in U\\
\iff& 01 010 01001 01001010\stkout{0100101001001 01001010(0100101001001 01001010)^{-1}}\varphi^4({\boldsymbol x})\in U\\
\iff& \varphi(00101001001\varphi^3({\boldsymbol x}))\in U\\
\iff& 00101001001\varphi^3({\boldsymbol x})\in U\\
\iff& \varphi(1001010\varphi^2({\boldsymbol x}))\in U\\
\iff& 1001010\varphi^2({\boldsymbol x})\in U\\
\iff& 1\varphi(1001\varphi({\boldsymbol x}))\in U\\
\iff& 01001\varphi({\boldsymbol x})\in U\text{ or }1001\varphi({\boldsymbol x})\in U_{001}\\
\iff& 01001\varphi({\boldsymbol x})\in U\\
\iff& \varphi(010{\boldsymbol x})\in U\\
\iff& 010{\boldsymbol x}\in U.
\end{align*}
Thus $(\beta\beta\beta)^{-1}W=(\beta\beta\beta\beta)^{-1}W$, as desired.

\item[(e)] We have \begin{align*}
&\beta\beta\alpha\beta{\boldsymbol f}\in W\\
\iff& 010\bullet\beta\beta\alpha\beta{\boldsymbol f}\in U\\
\iff& (010\bullet\beta\beta\alpha\beta)\varphi^4(F_2)^{-1}\varphi^4({\boldsymbol x})\in U\\
\iff& 01 01 0 01 0 01 01 0\stkout{0100101001001 01001010(0100101001001 01001010)^{-1}}\varphi^4({\boldsymbol x})\in U\\
\iff& \varphi(00101001\varphi^3({\boldsymbol x}))\in U\\
\iff& 00101001\varphi^3({\boldsymbol x})\in U\\
\iff& \varphi(10010\varphi^2({\boldsymbol x}))\in U\\
\iff& 10010\varphi^2({\boldsymbol x})\in U\\
\iff& 1\varphi(101\varphi({\boldsymbol x}))\in U\\
\iff& 0101\varphi({\boldsymbol x})\in U\text{ or }101\varphi({\boldsymbol x})\in U_{001}\\
\iff& 0101\varphi({\boldsymbol x})\in U\\
\iff& \varphi(00{\boldsymbol x})\in U\\
\iff& 00{\boldsymbol x}\in U.
\end{align*}
However, ${\boldsymbol x}$ has prefix $0$, so $00{\boldsymbol x}$ has the $4^-$-power $000$ as a prefix. Therefore, $00{\boldsymbol x}\not\in U.$ It follows that $(\beta\beta\alpha\beta)^{-1}W=\emptyset$.
\end{enumerate}
\end{proof}

\begin{remark} We mention without proof that $$\beta\alpha\bullet{\boldsymbol f}\in W\iff 1{\boldsymbol f}\in V,$$ and $$\beta\beta\bullet{\boldsymbol f}\in W\iff 10{\boldsymbol f}\in V.$$ We do not need these equivalences to formulate the automaton.
\end{remark}
\begin{remark}
As Fife's Theorem features a forbidden factor characterization, we have given such a characterization for $V$. In fact, however, all information about $V$ is captured in the automaton $A_{010}$. 
For a finite string $g\in{\mathcal O}^*$, word $010\bullet g$ is fb exactly when $g$ can be walked on the automaton $A_{010}$; such a string $g$ never encounters the (undepicted) non-acepting sink $(e)$, which can only be reached via $(c)$ on input $\beta$. If ${\boldsymbol f}\in{\mathcal O}^\omega$,  word $010\bullet {\boldsymbol f}$ is fb exactly when $010\bullet f$ is fb for finite prefix $f$ of  ${\boldsymbol f}$.

It is routine to write down an expression for the regular language of finite words arriving at the sink, which is $F{\mathcal O}^*$. It happens $F$ will take us from any given state to the sink.   For this reason, $010\bullet g$ is fb exactly when $g$ has no {\em factor} in $F$. 
\end{remark}

For $u\in\{010,0010,1010,10010\}$, let
$W_u = \{{\boldsymbol f}\in {\mathcal O}^\omega
: u\bullet{\boldsymbol f}\in U\}$. Using the same method as in Lemma~\ref{identities}, one shows that $W_u$ is the subset of ${\mathcal O}^\omega$ which can be walked on $A_u$, where the additional automata are depicted in Figures~\ref{automaton2}, \ref{automaton3}, and \ref{automaton4}.

\begin{figure}
\begin{center}
\begin{tikzpicture}[shorten >=1pt,node distance=100pt,auto]

  \tikzstyle{every state}=[fill={rgb:black,1;white,10}]

  \node[state,initial]   (a)                      {$a$};
  \node[state] (b) [below of=a]  {$b$};
  \node[state] (f) [right of=b] {$f$};
  \node[state] (g) [below left of=b]     {$g$};
  \node[state] (c) [below right of=b]     {$c$};
  \node[state] (d) [below left of=c]     {$d$};
  \path[->,arrows = {-Latex[width=7pt, length=7pt]},every loop/.style={arrows = {-Latex[width=7pt, length=7pt]}}]

  (a)   edge     node {$\alpha$} (f)
       edge  node {$\beta$} (g)
  (b) edge [bend left] node {$\alpha$} (f)
       edge  node {$\beta$} (g)
  (c) edge   node {$\alpha$} (b)
  (d) edge  node {$\alpha$} (c)
       edge [loop left]    node {$\beta$} (d)  
  (f) edge node {$\alpha$,$\beta$} (b)
  (g) edge  node {$\alpha$} (c)
       edge  node {$\beta$} (d)
 ;
\end{tikzpicture}
\caption{`Fife' automaton ${A_{0010}}$ for $W_{0010}$.}
\label{automaton2}
\end{center}
\end{figure}
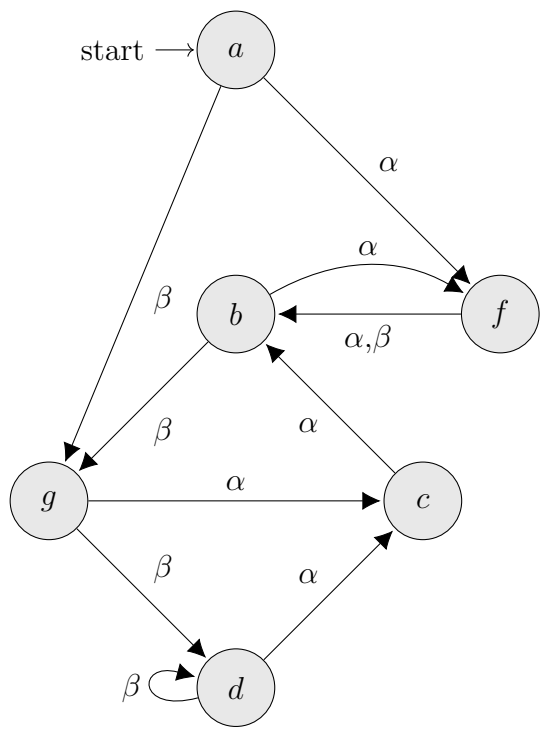

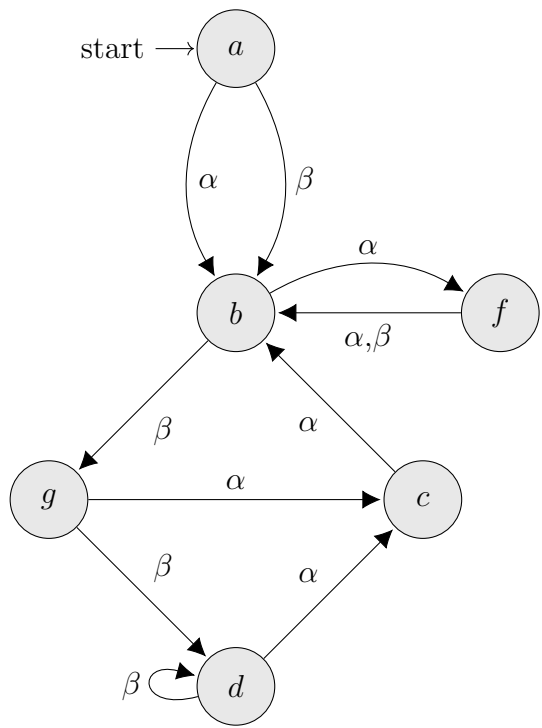
\begin{figure}
\begin{center}
\begin{tikzpicture}[shorten >=1pt,node distance=100pt,auto]

  \tikzstyle{every state}=[fill={rgb:black,1;white,10}]

  \node[state,initial]   (a)                      {$a$};
  \node[state] (b) [below of=a]  {$b$};
  \node[state] (f) [right of=b] {$f$};
  \node[state] (g) [below left of=b]     {$g$};
  \node[state] (c) [below right of=b]     {$c$};
  \node[state] (d) [below left of=c]     {$d$};
  \path[->,arrows = {-Latex[width=7pt, length=7pt]},every loop/.style={arrows = {-Latex[width=7pt, length=7pt]}}]

  (a)   edge [bend left] node {$\beta$} (b)
		edge [bend right] node {$\alpha$} (b)
  (b) edge [bend left] node {$\alpha$} (f)
       edge  node {$\beta$} (g)
  (c) edge   node {$\alpha$} (b)
  (d) edge  node {$\alpha$} (c)
       edge [loop left]    node {$\beta$} (d)  
  (f) edge node {$\alpha$,$\beta$} (b)
  (g) edge  node {$\alpha$} (c)
       edge  node {$\beta$} (d)
 ;
\end{tikzpicture}
\caption{`Fife' automaton ${A_{1010}}$ for $U_{1010}$.}
\label{automaton3}
\end{center}
\end{figure}

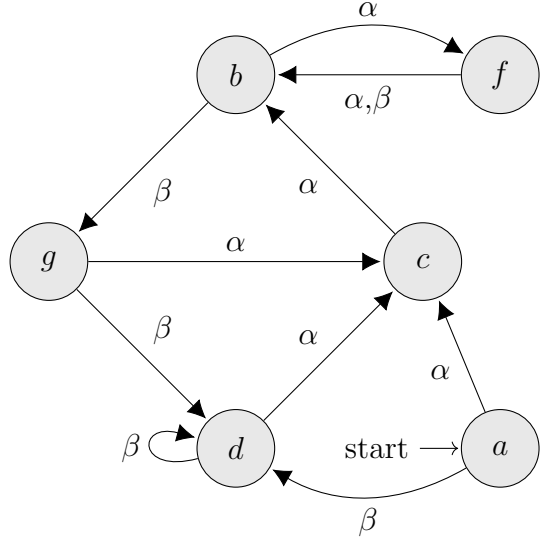
\begin{figure}
\begin{center}
\begin{tikzpicture}[shorten >=1pt,node distance=100pt,auto]

  \tikzstyle{every state}=[fill={rgb:black,1;white,10}]

  \node[state,initial]   (a)                      {$a$};
  \node[state] (d) [left of=a]     {$d$};
  \node[state] (g) [above left of=d]     {$g$};
  \node[state] (c) [above right of=d]     {$c$};
  \node[state] (b) [above right of=g]  {$b$};
  \node[state] (f) [right of=b] {$f$};

  \path[->,arrows = {-Latex[width=7pt, length=7pt]},every loop/.style={arrows = {-Latex[width=7pt, length=7pt]}}]

  (a)   edge   node {$\alpha$} (c)
       edge  [bend left] node {$\beta$} (d)
  (b) edge [bend left] node {$\alpha$} (f)
       edge  node {$\beta$} (g)
  (c) edge   node {$\alpha$} (b)
  (d) edge  node {$\alpha$} (c)
       edge [loop left]    node {$\beta$} (d)  
  (f) edge node {$\alpha$,$\beta$} (b)
  (g) edge  node {$\alpha$} (c)
       edge  node {$\beta$} (d)
 ;
\end{tikzpicture}
\caption{`Fife' automaton ${A_{10010}}$ for $U_{10010}$.}
\label{automaton4}
\end{center}
\end{figure}

\section{Lexicographically extremal fb words}
The lexicographic order on binary words is given recursively by
\[u< v\iff v\ne \epsilon\mbox{ and }((u=\epsilon)\mbox{ or }(u^-<v^-)\mbox{ or }((u=u^-0)\mbox{ and }(v=u^-1))).\]
Note that the morphism $\varphi$ is order-reversing: Let $u$ and $v$ be non-empty binary words so that $u<v$. Write $u=u'0u^{\prime\prime}$, $v=u'1v^{\prime\prime}$ where $u'$ is the longest common prefix of $u$ and $v$. Then $\varphi(u')01$ is a prefix of $\varphi(u)$, while $\varphi(u')00$ is a prefix of $\varphi(v)$, so that $\varphi(u)>\varphi(v)$. 

For each non-negative integer $n$, let $\ell_n$ (resp., $m_n$) be the lexicographically least (resp., greatest) word of length $n$ such that $\ell_n$ (resp., $m_n$) is the prefix of an fb $\omega$-word. 

\begin{lemma}\label{ell_n prefix} Let $n$ be a non-negative integer. Word $\ell_n$ is a prefix of $\ell_{n+1}$. Word $m_n$ is a prefix of $m_{n+1}$. 
\end{lemma}
\begin{proof} We prove the result for the $\ell_n$; the proof for the $m_n$ is similar. Let $\ell_n{\bf r}$ be an fb $\omega$-word. Let $p$ be the length $n+1$ prefix of
$\ell_n{\bf r}$, and let $q$ be the length $n$ prefix of
$\ell_{n+1}$. We need to show that $q=\ell_n$. Both $p$ and $q$ are prefixes of fb $\omega$-words. By definition we have $\ell_{n+1}\le p$ and $\ell_n\le q$. If $\ell_n<q$, then $p^-=\ell_n<q=\ell_{n+1}^-$, so that $p<\ell_{n+1}$. This is a contradiction. Therefore $\ell_n= q$, as desired.
\end{proof}

Let ${\bf {\boldsymbol\ell}}=lim_{n\rightarrow\infty}\ell_n$, ${\bf m}=lim_{n\rightarrow\infty}m_n$.

\begin{lemma} Word ${\bf {\boldsymbol\ell}}$ is the lexicographically least fb $\omega$-word. Word ${\bf m}$ is the lexicographically greatest fb $\omega$-word. 
\end{lemma}
\begin{proof} We show that ${\bf {\boldsymbol\ell}}$ is lexicographically least. The proof that ${\bf m}$ is lexicographically greatest is similar. Let ${\bf w}$ be an fb $\omega$-word. For each $n$ let $w_n$ be the length $n$ prefix of ${\bf w}$, so that ${\bf w}=\lim_{n\rightarrow\infty}w_n$. 

If for some $n$ we have $w_n>\ell_n$, then ${\bf w}>{\bf {\boldsymbol\ell}}$. 

Otherwise $w_n\le \ell_n$ for all $n$. By the definition of the $\ell_n$ we have  $w_n\ge \ell_n$, so that $w_n= \ell_n$ for all $n$. Thus ${\bf w}=\lim_{n\rightarrow\infty}w_n=\lim_{n\rightarrow\infty}\ell_n={\bf {\boldsymbol\ell}}$. 

In all cases we find ${\bf w}\ge{\bf {\boldsymbol\ell}}$.
\end{proof}

\begin{lemma}\label{least} We have ${\boldsymbol\ell}=\varphi({\bf m})$. 
\end{lemma}
\begin{proof} Since the Fibonacci word ${\boldsymbol \phi}$ has suffixes beginning with 00, $\ell_2=00$, and we can write ${\boldsymbol\ell}=\varphi({\bf m'})$ for some ${\bf m'}$ by Theorem~\ref{main theorem}. Since ${\boldsymbol\ell}$ is fb, ${\bf m'}$ is fb by Lemma~\ref{phi^-1 fb}. It follows that ${\bf m'}\le {\bf m}$. However if ${\bf m'}< {\bf m}$ then $\varphi({\bf m})<\varphi({\bf m'})={\boldsymbol\ell}$ since $\varphi$ is order-reversing. This is impossible, since ${\boldsymbol\ell}$ is least. Therefore ${\bf m'}= {\bf m}$, and ${\boldsymbol\ell}=\varphi({\bf m})$.
\end{proof}

\begin{lemma}\label{greatest} We have ${\bf m}=1\varphi({\boldsymbol\ell})$.
\end{lemma}
\begin{proof} Since 00 is a prefix of ${\boldsymbol \ell}$ we see that 0101 is a prefix of 
$\varphi({\boldsymbol\ell})$. It follows from Lemma~\ref{10101} that $1\varphi({\boldsymbol\ell})$ is fb.  Since 10101 is the lexicographically greatest fb word of length 5, $m_5=10101$. It follows that we can write ${\boldsymbol m}=
1\varphi({\boldsymbol \ell'})$ for some fb word ${\boldsymbol \ell'}$. However, $\varphi$ is order reversing, so that if 
${\boldsymbol \ell'}>{\boldsymbol \ell}$, then $1\varphi({\boldsymbol \ell})>1\varphi({\boldsymbol \ell'})={\boldsymbol m}$, contradicting the maximality of ${\boldsymbol m}$. Thus ${\boldsymbol \ell'}={\boldsymbol \ell}$ and ${\bf m}=1\varphi({\boldsymbol \ell})$.
\end{proof}

\begin{Theorem}\label{lex least} Word ${\boldsymbol m}$ satisfies
\begin{equation}\label{m}{\boldsymbol m}=1\varphi^2({\boldsymbol m}).\end{equation}
Word ${\boldsymbol \ell}$ satisfies
\begin{equation}{\boldsymbol \ell}=0\varphi^2({\boldsymbol \ell}).\end{equation}
\end{Theorem}
\begin{proof}
This follows from  Lemma~\ref{least} and Lemma~\ref{greatest}.
\end{proof}
\begin{lemma} Neither of ${\boldsymbol m}$ and ${\boldsymbol \ell}$ is the fixed point of a binary morphism. Every factor of ${\boldsymbol \phi}$ is a factor of ${\boldsymbol m}$ and ${\boldsymbol \ell}$, but there are infinitely many factors of ${\boldsymbol m}$ (resp., ${\boldsymbol \ell}$) which are not factors of ${\boldsymbol \phi}$ or ${\boldsymbol \ell}$ (resp., ${\boldsymbol m}$). \end{lemma}
\begin{proof} Word ${\boldsymbol m}$ has prefix 10101, but by Lemma~\ref{internal 10101}, the word 10101 is not a factor of $^-{\boldsymbol m}$. It follows that ${\boldsymbol m}$ cannot be the fixed point of a binary morphism. Similarly, $\varphi(10101)$ is a prefix of ${\boldsymbol \ell}$, but not a factor of $^-{\boldsymbol \ell}$, so that ${\boldsymbol \ell}$ is not a fixed point of a binary morphism.

Every factor of ${\boldsymbol \phi}$ is a factor of $\varphi^{2k}(0)$ for some $k$, and is therefore a factor of 
\[{\boldsymbol m}=1\varphi^2({\boldsymbol m})=1\varphi^2(1\varphi^2({\boldsymbol m}))=\cdots
=  1\varphi^2(1)\varphi^4(1)\cdots\varphi^{2k-2}(1)\varphi^{2k}({\boldsymbol m})\] 
Similarly, every factor of ${\boldsymbol \phi}$ is a factor of ${\boldsymbol \ell}$.
However, none of factors $\varphi^{2k}(10101)$ of ${\boldsymbol m}$ (resp., $\varphi^{2k+1}(10101)$ of ${\boldsymbol \ell}$) is a factor of ${\boldsymbol \phi}$ or ${\boldsymbol \ell}$ (resp, ${\boldsymbol m}$).
\end{proof}
\section{Acknowledgment}
The work of James D.\ Currie is supported by the Natural Sciences and Engineering Research Council of Canada (NSERC),
[funding reference number 2017-03901].
The work of Narad Rampersad is supported by the Natural Sciences and Engineering Research Council of Canada (NSERC),
[funding reference number 2019-04111].

\end{document}